\documentclass[12pt]{amsart} 
\usepackage{amsmath,amstext,amscd, amssymb, color, txfonts} 
\usepackage[colorlinks=true, pdfstartview=FitV, linkcolor=blue, citecolor=blue, urlcolor=blue]{hyperref}

\definecolor{cerulean}{rgb}{0,.48,.65} 
\definecolor{magenta}{rgb}{.5,0,.5} 
\definecolor{dred}{rgb}{.5,0,0} 
\definecolor{green}{rgb}{0,.5,0} 
\definecolor{blue}{rgb}{0,0,1} 
\definecolor{black}{rgb}{0,0,0} 
\definecolor{dgreen}{rgb}{0,.3,0} 
\definecolor{vdred}{rgb}{.3,0,0} 
\definecolor{red}{rgb}{1,0,0} 
\definecolor{salmon}{rgb}{0.98,0.50,0.45} 
\definecolor{gray}{rgb}{.5,.5,.5} 
\definecolor{seagreen}{rgb}{0.13,0.70,0.67} 
\definecolor{chartreuse}{rgb}{0.40,0.80,0.00}
\definecolor{cornflower}{rgb}{0.39,0.58,0.93} 
\definecolor{gold}{rgb}{0.80,0.68,0.00}
 
\theoremstyle{plain}

\newtheorem{thm}{Theorem}[section]
\newtheorem{lemma}[thm]{Lemma}

\newtheorem{cor}[thm]{Corollary}
\newtheorem{prop}[thm]{Proposition}

\theoremstyle{definition}

\newtheorem{Open questions}[thm]{Open questions}
\newtheorem{Open question}[thm]{Open question}
\newtheorem{Open problems}[thm]{Open problems}
\newtheorem{Open problem}[thm]{Open problem}

\def\Bbb{\mathbb}

\def\Z{\Bbb{Z}}
\def\N{\Bbb{N}}

\def\ni{\noindent}

\def\Dist{\hbox{\rm Dist}}

\def\F+L{\hbox{$\textup{F}\!_+\textup{L}$}}

\def\ssm{\smallsetminus}

\def\supp{\hbox{\rm supp}}

\def\onto{{\kern3pt\to\kern-8pt\to\kern3pt}}

\def\<{\langle}
\def\>{\rangle}
\def\|{{\ |\ }}

\newcommand{\set}[1]{\left\{#1\right\}}

\newcommand{\abs}[1]{\left|#1\right|}
\renewcommand{\ni}{\noindent}

\def\*{^{\star}}

\begin{document}

\title[Exponentially distorted subgroups in wreath products]{Exponentially distorted subgroups \\ in wreath products}

\author{T.\ R.\ Riley \vspace*{5mm} \\  In memoriam Peter M.\ Neumann 1940--2020 }

\date \today

\begin{abstract}
\ni  
We exhibit exponentially  distorted subgroups in  $\Z \wr ( \Z \wr \Z )$ and $\Z \wr F_2$.    \\
\ni \footnotesize{\textbf{2020 Mathematics Subject Classification:  20F65, 20F10, 20F16
}}  \\ 
\ni \footnotesize{\emph{Key words and phrases:} wreath product, subgroup distortion}
\end{abstract}

\maketitle


\section{Introduction}

The main result of this paper is --

\begin{thm} \label{main thm} 
The subgroup $$H = \langle \,  x, \,  y,  \, [x,a]a, \,  [y,a]a \, \rangle$$ is exponentially distorted in $\Z \wr F_2$ where $\Z = \langle{a}\rangle$ and $F_2 = F(x,y)$.  The same is true in   $\Z \wr ( \Z \wr \Z) = \langle a  \rangle \wr ( \langle s  \rangle \wr \langle t  \rangle)$ with  $x=ts$ and $y=t$.    
\end{thm}

Distortion of finitely generated subgroups in finitely generated groups is foundational and widely studied.  It compares a subgroup's word metric with the restriction of the word metric of the ambient group. For example, the subgroup $H = \langle a \rangle \cong \Z$ of  $G= \langle a, t  \mid t^{-1} a t = a^2 \rangle$ is said to be \emph{at least} exponentially distorted because for all $n \in \N$, the length-$2^n$ word $a^{2^n}$ equals the length-$(2n+1)$ word $t^{-n} a t^n$ in $G$; in fact, it \emph{is} exponentially distorted because,  moreover, there is a constant $C>0$ such that whenever a word  on $a$ and $t$ of length $k$  represents an element $a^L$ of $H$, we have $L \leq C^k$.  The Heisenberg group $G=\langle a ,b,c\mid [a,c], [b,c], [a,b]c^{-1} \rangle$ provides another example: its center    $\langle c \rangle \cong \Z$ is \emph{at least} quadratically distorted because $[a^n, b^n] = c^{n^2}$ in $G$ for all $n \in \N$; and, in fact, it \emph{is} quadratically distorted because, moreover, there exists $C>0$ such that for every word  on $a$, $b$, and $c$ of length $k$ that represents an element $c^L$ of $H$, we have $L \leq Ck^2$.

In some fundamental cases subgroup distortion is well-behaved. Subgroups of finitely generated free groups and of fundamental groups of closed hyperbolic  surfaces are undistorted \cite{Pittet2, Short2}. 
 Subgroups of finitely generated nilpotent groups are all at most polynomially distorted \cite{OsinDistNilp}.  But subgroup distortion can be wild even in some seemingly benign groups. There are subgroups of  $F_2 \times F_2$ and of rank-3 free solvable groups   whose distortion functions cannot be  bounded from above by a recursive function \cite{Mihailova, Umirbaev}.

Theorem~\ref{main thm}  is a next step in a direction of inquiry pursued by Davis and Olshanskii  \cite{DavisThesis, DO}.  They proved that every subgroup of $\Z \wr \Z$ is distorted like $n^d$ for some positive integer $d$ and, for all such $d$, they exhibited a subgroup  realizing that distortion.  Davis~\cite{DavisThesis} suggested next exploring subgroup distortion in $\Z \wr F_n$ and quoted speculation that an answer would be of interest for the study of von~Neumann algebras. Theorem~\ref{main thm} reveals a sharp contrast between subgroup distortion in $\Z \wr \Z$ and in $\Z \wr F_2$. Also, given that all finitely generated subgroups  in $\Z$ and in $F_2$ are undistorted and   in $\Z \wr \Z$ are at most polynomially distorted, it  shows that the wreath-product construction can give rise to substantial subgroup distortion.

The most novel feature of the work here is the idea behind the exponential lower bound (as proved in Section~\ref{exp lower bound}).  It relies on the observation that $F_2$ and   $\Z \wr \Z$ admit height functions (homomorphisms onto $\Z$) such that  for all integers $n>0$, there are pairs of height-$0$ elements a distance $2n$ apart with the property  that any path from one to the other travels up to height $n$ en route---see Proposition~\ref{pebbles prop}.

An intrinsic  description of the subgroup $H$ is not immediately evident from its definition in Theorem~\ref{main thm}.  Our proof of the exponential upper bound on its distortion in Section~\ref{exp upper bound}  includes such a  description.    

The second theorem of this article makes the point that our subgroups of Theorem~\ref{main thm} are necessarily delicate (given that  all subgroups are undistorted in $\Z$ and in $F_2$) in that exponentially distorted subgroups have to intersect both factors of the wreath product non-trivially and cannot be $\Z$-subgroups.  Closely related results can be found in \cite{BL-P}, which we recommend for a more detailed treatment than the proof we outline in Section~\ref{other subgroups}.

\begin{thm} (Cf.\ Burillo--L\'opez-Plat\'on \cite{BL-P})
\label{other subgroups thm}
Suppose $K$ is a finitely generated group  and $G = \Z \wr K$.  So, $G = W \rtimes K$ where $W =   \bigoplus_{K}  \Z$. Then --
\begin{enumerate}
\item All finitely generated subgroups $H$ of $W$ are undistorted in $G$ (meaning $\Dist^{G}_{H}(n) \simeq n$). \label{W case}
\item If $H$ is a finitely generated subgroup of $K$, then its distortion in $G$ is the same as its distortion in $K$ (more precisely, $\Dist^{G}_{H} \simeq \Dist^{K}_{H}$). In particular, $K$ is undistorted in $G$ (meaning $\Dist^{G}_{K}(n) \simeq n$). \label{K case}
\item Cases \eqref{W case} and \eqref{K case} give all possible distortion functions of $\Z$-subgroups of $G$.  In more detail, if $\hat{H} \cong \Z$ is a subgroup of $G$ then either $\hat{H}$ is undistorted in $G$ or there exists a subgroup $H  \cong \Z$  of $W$ or $K$  such that  $\Dist^{G}_{H} \simeq \Dist^{G}_{\hat{H}}$.  \label{Z case}
\end{enumerate}
\end{thm}

In the case of $G = \Z \wr F_2$ all the subgroups in this list are undistorted in $G$, because all finitely generated subgroups of $F_2$ are undistorted.  In the case of  $G = \Z \wr ( \Z \wr \Z )$ the list includes polynomially distorted subgroups on account of \cite{DavisThesis, DO}.
 
 It is tempting to try to use the results in this paper to address the question of Guba \& Sapir \cite{GubaSapir3} as to what functions may be distortion functions of finitely generated subgroups of Thompson's group $F$. However they do not speak to that because, while $( \Z \wr \Z ) \wr \Z$ is a subgroup of Thompson's group,  $\Z \wr ( \Z \wr \Z )$ and $\Z \wr F_2$ are not \cite[Theorem~1.2]{Bleak}.  

A companion to this paper is in preparation, proving that for $\Z \wr F_2$, Theorem~\ref{main thm} is best possible---that is, every finitely generated subgroup of $\Z \wr F_2$ is at most exponentially distorted \cite{BeauRi}.

  \section*{Acknowledgements}  I am most grateful to Aria Beaupr\'e, Jim Belk, Conan Gillis, and Chaitanya Tappu for helpful and stimulating conversations and I thank an anonymous referee for generously thoughtful feedback on the exposition, greatly enhancing this paper.   
 
 This paper is dedicated to the memory of Peter Neumann in tribute to his foundational work both as a scholar (e.g., pertinently, \cite{Neumann}) and as a teacher.

\section{Preliminaries}

For a group $G$   with finite generating set $S$, let $|g|_S$ denote the  length of a shortest word on $S^{\pm 1}$   representing $g$.   The word metric $d_S$ on $G$ is    $d_S(g, h) = \abs{g^{-1}h}_S$.  

Suppose  a subgroup $H \leq G$ is generated by  a finite  set  $T \subseteq G$.  The distortion function  $\Dist^{G}_{H} : \N \to \N$ for $H$ in $G$  compares  the word metric  $d_T$ on $H$ to the restriction of $d_S$ to $H$:   
$$\Dist^{G}_{H}(n) \ \coloneqq \  \max \set{  \  \abs{g}_T  \  \mid  \   g \in {H} \textup{ and }  \abs{g}_S  \leq n   \ }.$$
 
For functions $f, g : \N \to \N$ we write $f \preceq g$ when there exists $C>0$ such that $f(n) \leq C g( Cn +C) + Cn +C$ for all $n$.  We write $f \simeq g$ when $f \preceq g$ and $g \preceq f$.   
 
Two  finite  generating sets for a group yield biLipschitz word metrics, with the constants reflecting the minimal length words required to express the elements of one  generating set as words on the other.   So, up to $\simeq$, the growth rate of a distortion function does not depend on the finite generating sets.

Let $W =   \bigoplus_{K}  L$, the  direct sum of a  $K$-indexed family of copies of  $L$. The \emph{(restricted) wreath product} $G = L \wr K$ is the semi-direct product  $W \rtimes K$  with $K$ acting to shift the indexing.  More precisely, given  a function $f : K \to L$ that is  finitely supported (meaning $f(k)=e$ for  all but finitely many $k \in K$) and given $k \in K$, define $f^k : K \to L$ by $f^k(v) =  f(vk^{-1})$.   Then $L \wr K$ is the set of such pairs $(f,k)$ with  multiplication  $$(f,k)(\hat{f},\hat{k}) \  = \  (  f  + \hat{f}^k, k   \hat{k} ).$$

A \emph{lamplighter  description}  helps us navigate $G$.  Suppose $\set{a_1, \ldots, a_m}$ generates $L$ and $\set{b_1, \ldots, b_l}$ generates $K$.  Viewing $L$ and $K$ as subgroups of $G$, with $L$ being the $e$-summand of $W$, the set   $S = \set{a_1, \ldots, a_m, b_1, \ldots, b_l}$  generates  $G$.  Then $W$ is  the normal closure of $a_1, \ldots, a_m$ in $G$, or equivalently the kernel of the map $\Phi: G \onto K$ that kills $a_1, \ldots, a_m$.  
Imagine $K$ as a city.   At each street corner (that is, each element of $K$) there is a 
lamp whose setting is expressed as an element of $L$.  An   $(f,z) \in L \wr K$ records   settings $f(k) \in L$ of the lamps $k \in K$  and a  location $z \in K$ for the lamplighter.    A word $w$ on $S^{\pm 1}$ representing $(f,z)$ describes how at dusk a lamplighter walks the city streets  adjusting the  lamps to achieve $(f,z)$.  He starts at  $e \in K$ with all lights off (that is, set to $e \in L$) and, reading $w$ from left to right, moves in $K$ according to the $b_1^{\pm 1}, \ldots, b_l^{\pm 1}$ until finally arrives at $z$.  En route, he adjusts the setting of each lamp where he stands according to the $a_1^{\pm 1}, \ldots, a_m^{\pm 1}$.

Our conventions are that $[x,a] = x^{-1} a^{-1} xa$ and $x^a = a^{-1} x a$.

  \section{The exponential lower bound on distortion} \label{exp lower bound}

  \begin{prop} \label{pebbles prop}
  Suppose $K = \langle  x, y \mid R  \rangle$ is a 2-generator group such that mapping $x$ and $y$ to $1$ defines an epimorphism  $\theta: K \to \Z$ (a `height function').    
  
  Suppose that for $n \geq 1$, there is a set $P_n$ of  elements of $K$ such that:
  \begin{itemize}
  \item[\emph{(i)}] $x^{n-1} \in P_n$ but $x^ny^{-n} \notin P_n$.
  \item[\emph{(ii)}] If $p \in P_n$, then $p x^{-1}, p y^{-1} \in P_n$.   
  \item[\emph{(iii)}] If $k \in K \ssm P_n$ and either $k x^{-1}$ or $k y^{-1}$ is in $P_n$, then $\theta(k) =n$.
  \end{itemize}

  Let $G = \Z \wr K$,   generated by $a, x, y$ where $\Z = \langle a \rangle$.  Let $$H \ = \ \langle x, \ y, \  \sigma, \  \tau \rangle \ \leq G$$ where $\sigma = [x,a]a$ and $\tau = [y,a]a$.  Then  $\Dist^{G}_{H}(n) \succeq 2^n$.   
  \end{prop} 

For example,  when $K = F_2 = F(x,y)$, because the Cayley graph is a tree, the  proposition applies with $P_n$ the set of reduced words whose prefixes $\pi$ all satisfy $\theta(\pi) < n$.

We view \emph{(ii)} as saying that when moving in the Cayley graph of $K$, it is not possible to enter $P_n$  from below, and   \emph{(iii)} as saying that $P_n$ can only be entered from above by moving from a height-$n$ element outside $P_n$ to a height-$(n-1)$ element in $P_n$.       Together,  \emph{(i)} and \emph{(ii)} imply that  $x^i \in P_n$  if and only if $i <n$.   

Here is the idea behind this proposition in terms of the lamplighter description.  

Suppose the lights at the elements $e$ and $x^n y^{-n}$ of $K$ are set to $1$ and $-1$, respectively, and all other lights are off (set to $0$).  How can a lamplighter turn all the lights off using $x$, $y$, $\sigma$, and $\tau$? He has four types of moves at his disposal: he can navigate the Cayley graph of $K$ (by using $x$ and $y$); because $\sigma = [x,a]a = x^{-1} a^{-1} x a^2$, he can decrement by $1$ the lamp one step away in the $x^{-1}$-direction at the expense of incrementing the lamp where he stands by $2$; and likewise in the $y^{-1}$-direction using $\tau$.   The answer is he sets the lamp at $e$ to $0$ at the expense of setting the lamp  at $x$ to $2$.  Then he sets the lamp  at $x$ to $0$     at the expense of setting the lamp  at $x^2$ to $4$.  And so on, until the lamp at $x^n$ is set to $2^n$.  He then sets that to $0$ and, proceeding in the $y^{-1}$ direction, sets the lamp at $x^ny^{-1}$ to $2^{n-1}$.  Continuing likewise in the $y^{-1}$-direction he sets the lamp at $x^ny^{-(n-1)}$ to $2$.  Finally, he adjusts the lamp at $x^ny^{-(n-1)}$  to zero at the expense of changing the lamp at $x^ny^{-n}$, but as that was initially set to $-1$, this results in all lights being off, as required.  

The above method takes at least $2^n$ moves, but could  it have been accomplished with fewer?  The hypothesis involving $P_n$,  $x^n y^{-n}$, and the epimorphism  $K \to \Z$  ensures it cannot.  Any path from    $e$ to $x^n y^{-n}$ in the Cayley graph must rise to height $n$ to escape $P_n$ and the settings of the lights must be incrementally adjusted on the way up so that the number of $\sigma$- and $\tau$-moves grows exponentially with the height.          

Here is a proof.

 \begin{proof}[Proof of Proposition~\ref{pebbles prop}]  Fix $n \geq 1$. 
 First we   show that   $a^{-1} x^n y^{-n} a   \in H$. Define      
 \begin{eqnarray*} 
  \lambda_n  & = &       x  \sigma  \   x \sigma^2 \     \cdots \  x  \sigma^{2^{n-1}} \\ 
  \mu_n  & = &       y  \tau  \   y \tau^2 \    \cdots \  y  \tau^{2^{n-1}},  
      \end{eqnarray*} 
which   both represent elements of $H$.  In $G$, the elements $a$ and $x^{-1}a x$ commute,   so for all $i$,  $$a^i \ x \sigma^i  \ = \ a^i \ x \  (x^{-1} a^{-1} x a^2)^i \ = \   a^i \ x \   x^{-1} a^{-i} x  \ a^{2i}   \ = \ x \  a^{2i},$$ and therefore $a \lambda_n = x^n a^{2^{n}}$.  Likewise, $a \mu_n = y^n a^{2^{n}}$ in $G$. So $a^{-1} x^n y^{-n} a$ equals $\lambda_n \mu_n^{-1}$ in $G$ and represents an element of $H$.

 The length of $\lambda_n \mu_n^{-1}$ as a word on $x,   y,  \sigma, \tau$  is $2n + 2^{n+1} -2$.      Next we will argue that the length of \emph{any} word $w$ on $x,   y,  \sigma, \tau$  that represents  $a^{-1} x^n y^{-n} a$ in $G$ is at least $2^n -1$.  The length of  $a^{-1} x^{n} y^{-n} a$   as a word on $a,x,y$ is  $2n+2$.  So we will then have that 
$\Dist^{G}_{H}(2n+2) \geq 2^n -1$ and the  proposition will   follow.

Given a   finitely supported function   $h : K \to \Z$ and an integer  $i  < n$, define  
 $$p_i(h) \ = \  \sum_{\substack{g \in P_n \\ \theta(g) = i } } h(g).$$

Express $a^{-1} x^n y^{-n} a$ in the form $(f, x^n y^{-n})$ where $f$ is $-1$ at $e$, is $1$ at $x^n y^{-n}$, and is $0$ elsewhere.     Then the sequence 
\begin{equation*}
\mathcal{P}_n(f) \ = \  \left( \rule{0mm}{4mm} \ldots, \ p_{-1}(f), \ p_0(f), \ p_{1}(f), \  \ldots, \  p_{n-2}(f),  \ p_{n-1}(f)\right)
\end{equation*}
 is all zeroes apart from   $p_0(f)=-1$, since $e$ is in $P_n$ but $x^n y^{-n}$ is not.  
 
Now consider the effect on $\mathcal{P}_n(f)$ of changing $f$ via the action of $\sigma$ or $\tau$    
when the lamplighter is located at some  $k \in K$.  Let $i = \theta(k)$.   If $k \in P_n$, then (by hypothesis) $k x^{-1}$ and $k y^{-1}$ are in $P_n$ and so  $p_{i-1}(f)$ is lowered by $1$ and $p_{i}(f)$ is increased by $2$.  And if $k \notin P_n$, then  $k x^{-1}$ and $k y^{-1}$  can only be in  $P_n$ if $i=n$ (again, by hypothesis) and if so, $p_{n-1}(f)$ (only) decreases by $1$.  The effects of the actions of  $\sigma^{-1}$ and $\tau^{-1}$ are the same, but instead of lowering lamp settings by $1$ they increase them by $1$, and instead of increasing by $2$ they decrease by $2$.

We can read off   $w^{-1}$ a sequence of applications of $\sigma^{\pm 1}$ and $\tau^{\pm 1}$ (and  lamplighter movements around  $K$) that   convert  $a^{-1} x^n y^{-n} a$  to $e$ and so  convert $\mathcal{P}_n(f)$ to the sequence of all zeroes.  We will argue that this process must display a doubling effect that implies a  lower bound of $2^n-1$ on the length of $w$.  

For all integers $i < n$, let $\alpha^{+}_i$ (respectively, $\alpha^{-}_i$)  be the number of prefixes $k$ of $w^{-1}$ that have $\theta(k)=i$, have $k \in P_n$, and have final letter $\sigma$ or $\tau$ (respectively, $\sigma^{-1}$ or $\tau^{-1}$). Let $\alpha^{+}_n$ (respectively, $\alpha^{-}_n$)  be the number of prefixes $k$ of $w^{-1}$ that have $\theta(k)=n$, have $k \in P_n$, and either have final letter $\sigma$ (respectively, $\sigma^{-1}$) and $k x^{-1} \in P_n$,   or have final letter $\tau$ (respectively, $\tau^{-1}$) and $k  y^{-1} \in P_n$.  For all $i \leq n$, let $\alpha_i = \alpha^{+}_i - \alpha^{-}_i$.

 The net effect of $w^{-1}$ on $\mathcal{P}_n(f)$ is to  add $2\alpha_{i} - \alpha_{i+1}$ to the entry $p_i(f)$ for all $i < n$, because the order in which the applications of the relevant $\sigma^{\pm 1}$ and $\tau^{\pm 1}$ occur is immaterial. So, because $\mathcal{P}_n(f)$ is converted to the sequence of all zeros,  we have 
\begin{equation}
 	      2\alpha_{i} - \alpha_{i+1}     \ = \begin{cases} 0  & \text{ for }  i < n \text{ with } i \neq 0,  \label{alphas} \\ 
 1 & \text{ for }   i = 0.    
 \end{cases}  
\end{equation}

Now, $\alpha_i = 0$ for all $i$ sufficiently large and negative.  So  we deduce from \eqref{alphas} that $\alpha_i = 0$ for all $i \leq 0$, and then that $\alpha_1 = -1$, then $\alpha_2 = -2$, and so on until  $\alpha_{n} = -2^{n-1}$.    The sum $|\alpha_1| + \ldots + |\alpha_{n-1}| = 2^n-1$ is a lower bound on the number of letters $\sigma^{\pm 1}$ and $\tau^{\pm 1}$ in $w^{-1}$ and therefore on the length of $w$.
\end{proof}

  (In fact, when $K$ is $F_2$ or $\Z \wr \Z$, as per the following corollary, the roles of $x$ and $y$ are interchangeable and the above proof  shows $\lambda_n \mu_n^{-1}$ is a geodesic word.)

 \begin{cor} \label{cor specific examples}
 The subgroups of  $\Z \wr ( \Z \wr \Z )$ and  $\Z \wr F_2$ of Theorem~\ref{main thm}    are both at least exponentially distorted.
 \end{cor}

 \begin{proof}
 
How Proposition~\ref{pebbles prop} applies to $\Z \wr F_2$ was explained after its statement.

For  $\Z \wr ( \Z \wr \Z )$, consider  $K = \Z \wr \Z = \langle s, t \mid [s, s^{t^i}]=1  \ \forall i  \rangle$ and its generating set $x=ts$ and $y=t$.  Mapping  $x, y \mapsto 1$ defines an epimorphism $\theta: K \onto \Z$.   

In the lamplighter model for $\Z \wr \Z$, the integer $\theta(k)$ is the position of the lamplighter. Take $P_n$ to be the set of  all $k = (f, i) \in \Z \wr \Z$ such that  $\theta(k) = i < n$ and the $f$ is supported on $\{.... , n-2, n-1\}$.  Then $k x^{-1}$ and $k y^{-1}$ are in $P_n$ for all $k \in P_n$ since $x^{-1}$ decrements the light at the lamplighter's location and  then moves one step in the negative direction, and $y^{-1}$ only moves one step in the negative direction.  The elements $x^i$ for $0 \leq i \leq n-1$ are in $P_n$ because they set the lights at positions $1, 2, \ldots, i$  to $1$ and in all other positions to $0$  and they locate the lamplighter at position $i$.  However, $x^n y^{-n} = (ts)^n t^{-n}$  has the lights at positions $1, 2, \ldots, n$  set to $1$ (and at all others positions set to $0$), so is not in $P_n$.  And if $k \in K$ and either $k x^{-1}$ or $k y^{-1}$ is in $P_n$, then $\theta(k) =n$.
 \end{proof}

 The same proof works for  $\Z \wr ( C \wr \Z )$ for any finite cyclic group $C \neq \{ 1 \}$.  

An example where Proposition~\ref{pebbles prop} does not apply may be illuminating. The hypotheses on $P_n$ imply that any path from  $e$ to $x^n y^{-n}$ in the Cayley graph of $K$ must climb to height $n$ en route.   If $K = \Z^2 = \langle x,y \mid [x,y] \rangle$, then this need not happen, because  $x^n y^{-n} = (x y^{-1})^n$.  Indeed, in $\Z \wr \Z^2$ we find that $a^{-1} x^n y^{-n} a =  \left( a^{-1} xy^{-1} a \right)^n = ((x \sigma)(y \tau)^{-1})^n$, a word of length $4n$ on the generators of $H$.   
 
  \section{The exponential upper bound on distortion} \label{exp upper bound}
   
   Let $G = \Z \wr K$ where $K$ is $F_2$ or $\Z \wr \Z$ as per Theorem~\ref{main thm}.  
 Let $\theta: K \to \Z$ be the epimorphism mapping $x$ and $y$ to $1$.

\begin{lemma} 
	The subgroup  $H$ of $G$ of Theorem~\ref{main thm} is the set of all $g = (f,k) \in G$ such that 
	\begin{equation} \label{H lamp condition} \sum_{i \in \Z} 2^{-i} \sum_{\substack{v \in K, \  \theta(v) = i}} f(v) \ = \ 0. \end{equation}    
\end{lemma}
   
\begin{proof}
The four generators $x$, $y$, $\sigma = [x,a]a$ and  $\tau = [y,a]a$ of $H$ satisfy \eqref{H lamp condition}. And any   $g =   (f,k) \in G$ satisfying \eqref{H lamp condition} can be expressed as a word $u$ on $x^{\pm 1}$, $y^{\pm 1}$, $\sigma^{\pm 1}$, $\tau^{\pm 1}$ since it can can be transformed to the identity element as follows.  

Let $n = d_G(e, g)$, the length of the shortest word on $a^{\pm 1}, x^{\pm 1}, y^{\pm 1}$ representing $g$. The cardinality of $\supp f$ is at most $n$.   
Every   $h \in \supp f$ can be joined to $e$ in the Cayley graph of $K$  (with respect to $x$ and $y$) by a   path of length at most $n$.  The lamp setting $f(h)$ at $h$ has  absolute value at most $n$. By moving along this path (using $x^{\pm 1}$ and $y^{\pm 1}$) and successively adjusting  lamps along it (using $\sigma^{\pm 1}$ and $\tau^{\pm 1}$), the lamplighter can reset the lamp at $h$ to $0$ at the expense of changing the lamp at $e$ by at most $2^n$ while, in the process, the lamp settings always satisfy \eqref{H lamp condition}.  Once all the other lights have been extinguished the light at $e$ is also at $0$ on account of \eqref{H lamp condition}. \end{proof}

The above argument is  quantified in such a way that a couple of further observations complete the exponential upper bound proof for Theorem~\ref{main thm}.  
The absolute values of the settings of the lamps along the at most $n$ paths will grow to at most $n + n 2^n$ in the course of the transformation of the lamp settings.  The number of times $x^{\pm 1}$ and $y^{\pm 1}$ are used (for the movement) is at most $n^2$.   So $u$ has length at most a constant times $2^n$, establishing the exponential upper bound on $\Dist^{G}_{H}$.

 \section{Elementary subgroups of $\Z \wr K$}  \label{other subgroups}

Here we prove Theorem~\ref{other subgroups thm}. We have $G = \Z \wr K$, where $K$ is a finitely generated group. So $G = W \rtimes K$, where $W =   \bigoplus_{K}  \Z$.

Claim \eqref{W case} is that if  $H$ is a  finitely generated subgroup of $W$, then $H$ is undistorted in $G$.  Well, because  $H$ is finitely generated, it is a subgroup of the product of finitely many of the summands in $W =   \bigoplus_{K}  \Z$ and there exists $C \geq 1$ such that for all $g = (f,e) \in H$, both 
$d_G(e,g)$ and $d_H(e,g)$ (word metrics with respect to the generating sets for $G$ or for $H$, respectively) are between  $\frac{1}{C}\max_{i \in K} |f(i)|$ and $C\max_{i \in K} |f(i)|$.  So $H$ is undistorted in $G$.

Claim \eqref{K case} is that if $H$ is a finitely generated subgroup of $K$, then $\Dist^{G}_{H} \simeq \Dist^{K}_{H}$.  This is  straight-forward on account of the map $G \onto K$ killing $W$.  

Finally, claim \eqref{Z case} is that if $\hat{H} = \langle t \rangle$ is a $\Z$-subgroup  of $G$, then either $\hat{H}$ is undistorted in $G$ or there exists a subgroup $H  \cong \Z$  of $W$ or $K$  such that  $\Dist^{G}_{H} \simeq \Dist^{G}_{\hat{H}}$. 

Well, $t= (f,k)$ for some $f \in W$ and some $k \in K$.   If $k$ has finite order $r$, then $t^r = (f', e)$ for some $f' \in W$, and $H = \langle t^r \rangle$ is a subgroup of $W$ such that $\Dist^{G}_{H} \simeq \Dist^{G}_{\hat{H}}$.

Suppose, on the other hand, $k$ has infinite order.   Roughly speaking, we will show that for all $j$, either $t^j$ illuminates lights close to most of $e, k, \ldots, k^j$ and $\hat{H}$ is therefore undistorted in $G$, or it only illuminates lights close to $e$ and $k^j$, and   $\hat{H}$ is therefore distorted in $G$ similarly to $\langle (\mathbf{0}, k) \rangle$.       

Let $F: K \to \Z$ be the map $\sum_{i \in \Z} f^{k^i}$.  In terms of the lamplighter model, $F$ tells us the settings of the lights after the lamplighter acts per $f$ at $k^i$ for every  $i \in \Z$. As $k$ has  infinite order, $F$ is well-defined---for any $h \in K$,  $f^{k^i}(h) =0$ for all but finitely many $i$---but $F$ need not be finitely supported and so may not represent an element of $W$.  Indeed, $F$ is invariant under the action of $k$, so either $F = \mathbf{0}$ (the zero-map), or $F$ has infinite support.  

For $j \geq 1$, let $f_j = \sum_{i=0}^{j-1} f^{k^i}$, so that $t^j = ( f_j , k^j)$.

Let $L>0$ be sufficiently large that $\supp f \subset N_L(e)$---that is, the radius-$L$ neighbourhood  of $e$ in the Cayley graph of $K$ contains the support of $f$.  Then $\supp F \subseteq N_L( \langle k \rangle)$  and  $\supp f_j \subseteq N_L( \{  k^0, k^1, \ldots, k^j \})$ for all $j \geq 1$.

Because $k$ has infinite order, for all $R>0$, there exists $i$ such that $k^i, k^{i+1}, \ldots$ are a distance greater than $R$ from $e$ in the Cayley graph of $K$.  It follows that there exists  $C>0$ such that  for all $j > 0$,  the functions  $F$ and $f_j$ agree on   $\mathcal{N}_j := N_L( \{  k^C, k^{C+1}, \ldots, k^{j-C} \})$.  So, as $F$ is $k$-invariant, $f_j$ follows the same repeating pattern as $F$ along $\mathcal{N}_j$---more precisely,  the restrictions of  $f_j$ to  $N_L( k^C)$,  to $N_L( k^{C+1})$, \ldots, and to $N_L( k^{j-C})$ all agree after translations by successive powers of $k$.  And therefore, if  $\supp F \neq \emptyset$,    there exists $\lambda, \mu >0$ such that for all $j>0$ we have $d_G(e, t^j) \geq \lambda j - \mu$, because to achieve the element $t^j \in G$, the lamplighter must visit every one of these neighbourhoods.  So $\hat{H}$ is undistorted in $G$.  And if, on the other hand, $\supp F = \emptyset$, then there exists $\nu >0$ such that for all $j$ and all $g \in K \ssm N_{\nu}(\{ e, k^j\})$,  we have     $f_j(g) =e$.     So $d_G(e, t^j) \leq d_G(e, u^j) + C$  where $u=(\mathbf{0}, k)$.  So $\Dist^{G}_{H} \simeq \Dist^{G}_{\hat{H}}$    where $H := \langle  u \rangle$, which is a subgroup of $K$.

\bibliographystyle{alpha}
\bibliography{bibli}

\def\cprime{$'$}
\begin{thebibliography}{BLP15}

\bibitem[Ble08]{Bleak}
C.~Bleak.
\newblock A geometric classification of some solvable groups of homeomorphisms.
\newblock {\em J. Lond. Math. Soc. (2)}, 78(2):352--372, 2008.

\bibitem[BLP15]{BL-P}
J.~Burillo and E.~L{\'o}pez-Plat{\'o}n.
\newblock Metric properties and distortion in wreath products, 2015.
\newblock
  \texttt{\href{https://arxiv.org/abs/1506.06935}{arxiv.org/abs/1506.06935}}.

\bibitem[BR]{BeauRi}
A.~Beaupr\'e and T.~R. Riley.
\newblock Subgroup distortion in wreath products of abelian groups with free
  groups.
\newblock In preparation.

\bibitem[Dav11]{DavisThesis}
T.~C. Davis.
\newblock {\em Subgroup distortion in metabelian and free nilpotent groups}.
\newblock PhD thesis, Vanderbilt University, 2011.

\bibitem[DO11]{DO}
T.~C. Davis and A.~Yu. Olshanskii.
\newblock Subgroup distortion in wreath products of cyclic groups.
\newblock {\em J. Pure Appl. Algebra}, 215(12):2987--3004, 2011.

\bibitem[GS99]{GubaSapir3}
V.~S. Guba and M.~V. Sapir.
\newblock On subgroups of the {R}. {T}hompson group {$F$} and other diagram
  groups.
\newblock {\em Mat. Sb.}, 190(8):3--60, 1999.

\bibitem[Mih66]{Mihailova}
K.A. Mihailova.
\newblock The occurence problem for direct products of groups.
\newblock {\em Mat. Sbornik}, 70:241--251, 1966.

\bibitem[Neu64]{Neumann}
P.~M. Neumann.
\newblock On the structure of standard wreath products of groups.
\newblock {\em Math. Z.}, 84:343--373, 1964.

\bibitem[Osi01]{OsinDistNilp}
D.~V. Osin.
\newblock Subgroup distortions in nilpotent groups.
\newblock {\em Comm. Algebra}, 29(12):5439--5463, 2001.

\bibitem[Pit93]{Pittet2}
Ch. Pittet.
\newblock Surface groups and quasi-convexity.
\newblock In {\em Geometric group theory, {V}ol. 1 ({S}ussex, 1991)}, volume
  181 of {\em London Math. Soc. Lecture Note Ser.}, pages 169--175. Cambridge
  Univ. Press, Cambridge, 1993.

\bibitem[Sho91]{Short2}
H.~Short.
\newblock Quasiconvexity and a theorem of {H}owson's.
\newblock In {\em Group theory from a geometrical viewpoint ({T}rieste, 1990)},
  pages 168--176. World Sci. Publ., River Edge, NJ, 1991.

\bibitem[Umi95]{Umirbaev}
U.~U. Umirbaev.
\newblock The occurrence problem for free solvable groups.
\newblock {\em Algebra i Logika}, 34(2):211--232, 243, 1995.

\end{thebibliography}


\ni  \textsc{Timothy R.\ Riley} \rule{0mm}{6mm} \\
Department of Mathematics, 310 Malott Hall,  Cornell University, Ithaca, NY 14853, USA \\ \texttt{tim.riley@math.cornell.edu}

      \end{document}